\documentclass[11pt, reqno]{amsart}

\title{On a question of Ruzsa about densities of sumsets}

\usepackage[T1]{fontenc}
\usepackage{amsmath}
\usepackage{amssymb}
\usepackage{amsthm}
\usepackage[left=3.5cm, right=3.5cm, paperheight=11.8in]{geometry}
\usepackage{hyperref}
\usepackage{fancyhdr}
\usepackage{enumitem}
\usepackage{comment}
\usepackage{nicefrac}
\usepackage{mathrsfs}
\usepackage{bm}
\usepackage{graphicx}
\usepackage[utf8]{inputenc}
\usepackage{cancel}
\usepackage{mathtools}

\AtBeginDocument{%
   \def\MR#1{}
}

\newtheorem{thm}{Theorem}[section]
\newtheorem{cor}[thm]{Corollary}
\newtheorem{lem}[thm]{Lemma}
\newtheorem{prop}[thm]{Proposition}

\theoremstyle{definition}
\newtheorem{defi}[thm]{Definition}
\let\olddefi\defi
\renewcommand{\defi}{\olddefi\normalfont}
\newtheorem{example}[thm]{Example}
\let\oldexample\example
\renewcommand{\example}{\oldexample\normalfont}
\newtheorem{rmk}[thm]{Remark}
\let\oldrmk\rmk
\renewcommand{\rmk}{\oldrmk\normalfont}

\newtheorem{question}[thm]{Question}
\let\oldquestion\question
\renewcommand{\question}{\oldquestion\normalfont}

\keywords{Sumsets; asymptotic density; upper and lower asymptotic densities; probabilistic method.}

\subjclass[2020]{Primary 11B13; Secondary 11B05, 11B30.}

\hypersetup{
    pdftitle={},
    pdfauthor={Paolo Leonetti},
    pdfmenubar=false,
    pdffitwindow=true,
    pdfstartview=FitH,
    colorlinks=true,
    linkcolor=blue,
    citecolor=green,
    urlcolor=cyan
}

\newcommand{\ve}{\varepsilon}
\renewcommand{\ast}{\star}

\author[P.~Leonetti]{Paolo Leonetti}
\address{
Universit\'{a} degli Studi dell’Insubria\\ via Monte Generoso 71 \\ Varese 21100\\ Italy}
\email{leonetti.paolo@gmail.com}
\urladdr{\url{https://sites.google.com/site/leonettipaolo}}

\begin{document}

\begin{abstract}
Let $\mathsf{d}$, $\mathsf{d}^{\ast}$, and $\mathsf{d}_{\ast}$ denote, respectively, the asymptotic density, the upper asymptotic density, and the lower asymptotic density on $\mathbb N$. We show that there are no  constants $0<\nu<1$ and $c>0$ such that
$$
\mathsf{d}_{\ast}(A+A)\ge c\bigl(\mathsf{d}^{\ast}(A+A)\bigr)^{1-\nu}\mathsf{d}(A)^{\nu}
$$
for every set $A\subseteq \mathbb{N}$ which admits asymptotic density. 
This answers in the negative a question of I.Z. Ruzsa. 
\end{abstract}
\maketitle
\thispagestyle{empty}

\section{Introduction}

For each subset $A\subseteq\mathbb N:=\{0,1,2,\ldots\}$, denote its lower asymptotic density and upper asymptotic density, respectively, by 
$$
\mathsf{d}_{\ast}(A):=\liminf_{N\to\infty}\frac{|A\cap[0,N)|}{N}
\quad \text{ and }\quad 
\mathsf{d}^{\ast}(A):=\limsup_{N\to\infty}\frac{|A\cap[0,N)|}{N}.
$$
If $\mathsf{d}_{\ast}(A)=\mathsf{d}^{\ast}(A)$, then we say that $A$ admits asymptotic density, and we denote the common value by $\mathsf{d}(A)$. 
Let $\mathcal{D}$ be the family of sets $A\subseteq \mathbb{N}$ which admit asymptotic density. 
The sumset of two sets $A,B \subseteq \mathbb{N}$ is denoted by $$A+B:=\{x+y:x\in A,\ y\in B\}.$$ 
Many relationships between the [upper or lower] asymptotic density of $A$ and its sumset $A+A$ have been obtained in the literature, see e.g. \cite{MR4986634, MR4080054, MR2167721, MR4474053, MR2244800, MR2813407, MR56632, MR4439850, MR90618}.

Hegyv{\'a}ri, Hennecart, and Pach mentioned in \cite[Section 2]{MR3939472} a question of I.Z. Ruzsa: 
\begin{question}\label{q:ruzsa}
    Do there exist constants $0<\nu<1$ and $c>0$ such that
\begin{equation}\label{eq:claimedRuzsa}
\mathsf{d}_{\ast}(A+A)\ge c\bigl(\mathsf{d}^{\ast}(A+A)\bigr)^{1-\nu}\mathsf{d}(A)^{\nu}
\end{equation}
for all $A \in \mathcal{D}$?
\end{question}
They also added that Ruzsa proved (in an unpublished manuscript) that, in case of an affirmative answer, we necessarily
have $\nu\ge \nicefrac{1}{2}$.

In \cite[Proposition~2.2]{MR3939472}, the authors 
constructed a set $A\in \mathcal{D}$ with
$$
\mathsf d(A)=\frac37,
\qquad
\mathsf d_{\ast}(A+A)=\frac67,
\quad \text{ and }\quad 
\mathsf d^{\ast}(A+A)=1,
$$
thereby showing that $A \in \mathcal{D}$ does not necessarily imply $A+A \in \mathcal{D}$. 
Question \ref{q:ruzsa} asks whether, despite such oscillation, the lower asymptotic density of $A+A$ must still be quantitatively controlled by the density of $A$ together with the upper asymptotic density of $A+A$. 
It is remarkable that by Kneser's theorem \cite{MR56632} the  inequality 
$
\mathsf d_{\ast}(A+A)
<
2\mathsf d_{\ast}(A)
$ 
implies that $A+A$ is eventually periodic (hence $A+A \in \mathcal{D}$). It follows that 
$$
\mathsf d_{\ast}(A+A)\ge \min\{2\mathsf{d}(A), \mathsf{d}^{\ast}(A+A)\}
$$
for all $A \in \mathcal{D}$. 

\begin{thm}\label{thm:construction}
Fix an integer $r\ge2$ and a real number $\delta\in(0,\nicefrac{1}{4})$. Then there exists a set $A\in \mathcal{D}$ such that
\begin{equation}\label{eq:mainclaim}
\mathsf{d}(A)=\delta^r, \quad 
\mathsf{d}^{\ast}(A+A)=1,
\quad \text{ and }\quad 
\mathsf{d}_{\ast}(A+A)
\le
(r+1)(2\delta)^{r-1}.
\end{equation}
\end{thm}

In the proof of Theorem \ref{thm:construction}, the set $A$ will be obtained by concatenating suitable long periodic pieces. Some pieces will make the sumset $A+A$ almost covering infinitely many large intervals, while other pieces will have the opposite behaviour keeping a sufficiently small lower asymptotic density. 

As an application of Theorem \ref{thm:construction}, we obtain a negative answer to Question \ref{q:ruzsa}: 
\begin{cor}\label{cor:ruzsa}
    There are no constants $0<\nu<1$ and $c>0$ such that Inequality \eqref{eq:claimedRuzsa} holds for all $A\in \mathcal{D}$.
\end{cor}

To state our last result, for each $x,y> 0$ and $q \in \mathbb{R}$ define the $q$-mean 
$$
\mathscr{M}_q(x,y):=\left(\frac{x^q+y^q}{2}\right)^{1/q},
$$
with the convention that $\mathscr{M}_0(x,y):=\sqrt{xy}$. It is well known that $(\mathscr{M}_q(x,y): q \in \mathbb{R})$ is an increasing net, with infimum $\min\{x,y\}$ and supremum $\max\{x,y\}$. As a consequence of Corollary \ref{thm:construction} (choosing $\nu=\nicefrac{1}{2}$), for each $c>0$ the inequality 
$
\mathsf{d}_\star(A+A) \ge c \mathscr{M}_0(\mathsf{d}(A),\mathsf{d}^\star(A+A)) 
$ 
cannot hold for all $A \in \mathcal{D}$. As we show below, the index $q=0$ is, in a sense, optimal: 
\begin{prop}\label{prop:optimality}
    Fix a real $q<0$ and define $c_q:=\min\{1,2^{1+1/q}\}$. Then 
    $$
    \mathsf{d}_\star(A+A) \ge c_q \mathscr{M}_q(\mathsf{d}(A),\mathsf{d}^\star(A+A))
    $$
    for all $A \in \mathcal{D}$ with $\mathsf{d}(A)>0$. 
\end{prop}

The proofs of 
our results 
will be given in Section \ref{sec:proofs}.

\section{Preliminary results}\label{sec:preliminaries}

The long pieces on which the sumset $A+A$ is almost covering will be obtained from the following elementary application of the probabilistic method. 

\begin{lem}\label{lem:randombasis11111}
Let $\alpha\in(0,1)$ and $\ve>0$. For every sufficiently large integer $H$, there exists a set $R\subseteq\mathbb Z/H\mathbb Z$ such that
$$
\left|\frac{|R|}{H}-\alpha\right|<\ve
\quad \text{ and }\quad 
R+R=\mathbb Z/H\mathbb Z.
$$
\end{lem}
\begin{proof}
Choose a random subset $R\subseteq\mathbb Z/H\mathbb Z$ by putting each residue into $R$ independently with probability $\alpha$. Define $X:=|R|$, so that 
$$
\mathbb EX=\alpha H
\quad \text{ and }\quad 
\mathbb E\bigl((X-\mathbb EX)^2\bigr)=H\alpha(1-\alpha).
$$
By applying Markov's inequality, we obtain that 
$$
\mathrm{Pr}\bigl(|X-\alpha H|\ge\ve H\bigr)
\le
\frac{\mathbb E((X-\mathbb EX)^2)}{\ve^2H^2}
=
\frac{\alpha(1-\alpha)}{\ve^2H}.
$$
In particular, the probability that the density condition fails tends to $0$ as $H\to\infty$.

Fix now a residue $s\in\mathbb Z/H\mathbb Z$, and consider the map $x\mapsto s-x$ on $\mathbb Z/H\mathbb Z$. Its fixed points are the solutions of $2x=s$. There are at most two such solutions. All other residues are partitioned into disjoint two-element sets of the form $\{x,s-x\}$. 
Consequently, there are at least $(H-2)/2$ pairwise disjoint two-element sets of this form. For each one of them, the probability that both elements belong to $R$ is $\alpha^2$. Since the pairs are disjoint and membership of distinct residues is independent, the corresponding events are independent. Therefore
$$
\mathrm{Pr}(s\notin R+R)
\le
(1-\alpha^2)^{(H-2)/2}
$$
Using also the elementary inequality $1-t\le e^{-t}$ for $t\ge0$, it follows that 
$$
\mathrm{Pr}(R+R\ne\mathbb Z/H\mathbb Z)
\le
\sum_{s \in \mathbb Z/H\mathbb Z}\mathrm{Pr}(s\notin R+R) 
\le 
H\exp\left(-\frac{\alpha^2(H-2)}2\right).
$$
Observe that the right-hand side tends to $0$ as $H\to\infty$. 

Hence, for all sufficiently large $H$, the sum of the probabilities of the two events $|X-\alpha H|\ge\ve H$ and $R+R\ne\mathbb Z/H\mathbb Z$ is smaller than $1$. For such an $H$, there is a realization of $R$ for which neither events occur. 
\end{proof}

As usual, for all $k \in \mathbb{Z}$ and $A\subseteq \mathbb{Z}$, we define $k\cdot A:=\{ka: a \in A\}$ and $A+k:=A+\{k\}$. 
\begin{lem}\label{lem:loudpiece111111}
Pick positive integers $x,y,H$ such that $y-x>2H$. Fix also $R\subseteq\mathbb Z/H\mathbb Z$ such that 
$R+R=\mathbb Z/H\mathbb Z$ and define 
$$
S:=(H\cdot \mathbb{Z}+R) \cap [x,y).
$$
Then $[2x+2H,2y-2H) \cap \mathbb{N}\subseteq S+S$. 
\end{lem}
\begin{proof}
Fix an integer $n \in [2x+2H,2y-2H)$. Since $R+R=\mathbb Z/H\mathbb Z$, there exist $r,t\in R$ such that $r+t\equiv n\bmod H$. 
We search for an integer $a\equiv r\bmod H$ such that $x\le a<y$ and $x\le n-a<y$, or, equivalently, 
$$
\max\{x,n-y+1\}\le a\le\min\{y-1,n-x\}.
$$
If $n\le x+y-1$, then this interval is $[x,n-x]$, and it contains 
$n-2x+1\ge 2H+1$ integers. 
If $n> x+y-1$, then this interval is $[n-y+1,y-1]$, and it contains $2y-1-n\ge2H$ integers. 
Thus, in either case, the interval contains at least $H$ consecutive integers. Hence it contains an integer $a$ satisfying $a\equiv r\bmod H$. This completes the proof since $b:=n-a \equiv n-r\equiv t \bmod{H}$ and $a,b \in [x,y)$, which implies $n \in S+S$. 
\end{proof}

Lastly, we will also use the following elementary estimate. 
\begin{lem}\label{lem:productestimate1111}
If $x_1,\ldots,x_r,y_1,\ldots,y_r\in[0,1]$, then
$$
\left|\prod_{i=1}^r x_i-\prod_{i=1}^r y_i\right|
\le
\sum_{i=1}^r|x_i-y_i|.
$$
\end{lem}
\begin{proof}
Observe the following identity
$$
\prod_{i=1}^r x_i-\prod_{i=1}^r y_i
=
\sum_{j=1}^r
\left((x_j-y_j)\left(\prod_{i<j}x_i\right)
\left(\prod_{i>j}y_i\right)\right).
$$
Since all the factors other than $x_j-y_j$ have absolute value at most $1$, the claim follows by the triangle inequality. 
\end{proof}


\section{Proofs}
\label{sec:proofs}


For the sake of clarity, we divide the proof of Theorem \ref{thm:construction} into several steps. Here, empty unions, intersections, sums, and products are understood
to be equal to $\emptyset$, $\mathbb N$, $0$, and $1$, respectively.
\begin{proof}
    [Proof of Theorem \ref{thm:construction}] 
    Throughout, pick a sequence $(\varepsilon_m)_{m\ge 1}$ with values in $(0,\nicefrac{1}{4})$ such that $\lim_m \varepsilon_m=0$. For each $m\ge 1$, pick also 
    primes $p_{m,1},\ldots,p_{m,r}$ such that  
    $$
    \max\left\{\frac{2}{\delta}, \frac{r}{\varepsilon_m}\right\}< p_{m,1}<p_{m,2}<\cdots<p_{m,r}.
    $$
    and set $\pi_m:=p_{m,1}p_{m,2}\cdots p_{m,r}$. 
    In addition, for each $m\ge 1$, applying Lemma \ref{lem:randombasis11111} with $\alpha=\delta^r$ and error $\varepsilon=\varepsilon_m$, it is possible to choose a sufficiently large integer $H_m$ and a set $R_m\subseteq\mathbb Z/H_m\mathbb Z$ such that
\begin{equation}\label{eq:choiceRm}
\left|\frac{|R_m|}{H_m}-\delta^r\right|<\varepsilon_m
\quad \text{ and }\quad 
R_m+R_m=\mathbb Z/H_m\mathbb Z.
\end{equation}

    \medskip
    
    \textbf{Auxiliary sequences.} We start constructing recursively auxiliary sequences $(x_m)_{m\ge 1}$ and  $(y_m)_{m\ge 1}$ of positive integers, vectors of integers $(z_{m,0},z_{m,1}, \ldots,z_{m,r})$, and vectors of primes $(q_{m,1},q_{m,2},\ldots,q_{m,r})$. We proceed by induction as it follows: 
    \begin{list}{$\bullet$}{}
    \item [(i)] Fix an integer $x_1\ge  H_1/\varepsilon_1$. 
    \item [(ii)] Suppose that $x_m$ has been defined for some $m\ge 1$. Then pick an integer 
    $$
    y_m \ge \frac{1}{\varepsilon_m}{\max \left\{x_m,  \pi_m \right\}}.
    $$
    \item [(iii)] Set $z_{m,0}:=y_m$. 
    Suppose that $z_{m,0},\ldots,z_{m,j-1}$ and $q_{m,1},\ldots,q_{m,j-1}$ have been defined for some $j \in \{1,\ldots,r\}$. Then by Dirichlet theorem on arithmetic progressions, see e.g. \cite{MR568909}, choose a prime $q_{m,j} \ge 2z_{m,j-1}/\delta$ such that 
    $$
    q_{m,j}\equiv 1\bmod{\pi_mq_{m,1}q_{m,2}\cdots q_{m,j-1}},
    $$
    and, then, an integer 
    $$
    z_{m,j} \ge \frac{1}{\varepsilon_m}\max\{z_{m,j-1}, \pi_mq_{m,1}q_{m,2}\cdots q_{m,j}\}.
    $$
    \item [(iv)] Suppose that $z_{m,r}$ has been defined for some $m\ge 1$. Then pick an integer 
    $$
    x_{m+1} \ge \frac{1}{\varepsilon_{m+1}}\max\left\{z_{m,r},H_{m+1}\right\}.
    $$
    \end{list}

\medskip

\textbf{Construction of $A$.} By the above definitions, we have $x_1<y_1<x_2<y_2<\cdots$ and $y_m=z_{m,0}<z_{m,1}<\cdots<z_{m,r}<x_{m+1}$ for all $m\ge 1$. For notational convenience, for each $m\ge 1$, set 
$z_{m,r+1}:=x_{m+1}$. 
In addition, for each $m\ge 1$ and $j \in \{0,1,\ldots,r\}$,  define the sets 
$$
B_m:=(H_m \cdot \mathbb{N}+R_m) \cap [x_m,y_m)
$$
and 
$$
C_{m,j}:=
\left(
\bigcap_{i=1}^j(q_{m,i}\cdot\mathbb N+W_{q_{m,i}})
\cap
\bigcap_{i=j+1}^r(p_{m,i}\cdot\mathbb N+W_{p_{m,i}})
\right)
\cap[z_{m,j},z_{m,j+1}),
$$
where $W_h:=\{0,1,\ldots,\lfloor \delta h\rfloor\}$ for each $h \ge 1$. 
To complete the construction, define 
$$
A:=\bigcup_{m\ge 1}(B_m \cup C_m),
$$
where $C_m:=C_{m,0}\cup C_{m,1}\cup \cdots \cup C_{m,r}$ for all $m\ge 1$. 

\medskip

\textbf{Density of $A$.} 
Define $\alpha:=\delta^r$. We will show that $A$ admits asymptotic density (that is, $A \in \mathcal{D}$) and that
$
\mathsf{d}(A)=\alpha.
$ 

Observe by construction that $A \cap [x_m,y_m)=B_m$ for all $m\ge 1$, hence on such interval $A$ agrees with the periodic set $H_m\cdot\mathbb N+R_m$. Thanks to \eqref{eq:choiceRm} and the definition of the auxiliary sequences, we have 
\begin{equation}\label{eq:observ1}
\left|\frac{|R_m|}{H_m}-\alpha\right|<\varepsilon_m
\quad \text{ and }\quad 
H_m\leq\varepsilon_m x_m.
\end{equation}

\bigskip

Fix now $m\geq1$ and $j\in\{0,1,\ldots,r\}$. Then the set $C_{m,j}$, which is periodic on the interval $[z_{m,j}, z_{m,j+1})$, has period
$$
T_{m,j}:=
\prod_{i=1}^{j}q_{m,i}\,\cdot 
\prod_{i=j+1}^{r}p_{m,i}.
$$
Observe that all the primes $q_{m,1},\ldots,q_{m,j},p_{m,j+1},\ldots,p_{m,r}$ are distinct and greater than $r/\varepsilon_m$. 
It follows from the Chinese remainder theorem that the density $C_{m,j}$ is 
$$
\rho_{m,j}:=
\prod_{i=1}^{j}\frac{|W_{q_{m,i}}|}{q_{m,i}} \cdot 
\prod_{i=j+1}^{r}\frac{|W_{p_{m,i}}|}{p_{m,i}}.
$$
Since $|W_h|=\lfloor\delta h\rfloor+1$ for each $h\ge 1$, we have
$$
0<
\frac{|W_h|}{h}-\delta
=
\frac{\lfloor\delta h\rfloor+1-\delta h}{h}
\leq
\frac{1}{h}.
$$
In particular, for all $h \in \{q_{m,1},\ldots,q_{m,j},p_{m,j+1},\ldots,p_{m,r}\}$, we obtain 
$$
\left|\frac{|W_h|}{h}-\delta\right|
\leq\frac{\varepsilon_m}{r}.
$$
Applying Lemma~\ref{lem:productestimate1111}, we obtain
\begin{equation}\label{eq:observ11}
|\rho_{m,j}-\delta^r|
\leq
\sum_{i=1}^{j}
\left|
\frac{|W_{q_{m,i}}|}{q_{m,i}}-\delta
\right|
+
\sum_{i=j+1}^{r}
\left|
\frac{|W_{p_{m,i}}|}{p_{m,i}}-\delta
\right| \le r\cdot \frac{\varepsilon_m}{r}=\varepsilon_m.
\end{equation}

\medskip

We next estimate the periods $T_{m,j}$. Since $z_{m,0}=y_m$, the choice of $y_m$ gives
\begin{equation}\label{eq:observ2}
T_{m,0}
=
\pi_m
\leq
\varepsilon_m y_m
=
\varepsilon_m z_{m,0}.
\end{equation}
In addition, for every $j\in\{1,\ldots,r\}$, we have
\begin{equation}\label{eq:observ3}
T_{m,j}
=
q_{m,1}\cdots q_{m,j}
p_{m,j+1}\cdots p_{m,r} 
\le \pi_mq_{m,1}\cdots q_{m,j} 
\leq\varepsilon_m z_{m,j}.
\end{equation}

\medskip

To sum up, on every subinterval of $[x_m,x_{m+1})$ in the construction of $A$,
the corresponding periodic sets have density within $\varepsilon_m$ of $\alpha$ and period at most $\varepsilon_m$ times the left endpoint of that interval.

\medskip

We now relabel by $(\xi_\ell)_{\ell\ge 1}$ the increasing sequence 
$$
x_1<y_1<z_{1,1}<\cdots<z_{1,r}<x_2
<y_2<z_{2,1}<\cdots.
$$
It follows by construction that $\xi_{\ell+1}>4\xi_\ell$ for all $\ell\ge 1$. For each $\ell\geq1$, let $\kappa(\ell)$ be the unique positive
integer such that 
$
x_{\kappa(\ell)}
\leq
\xi_\ell
<
x_{\kappa(\ell)+1}.
$ 

\medskip

On every interval $[\xi_\ell,\xi_{\ell+1})$, the set $A$ is a finite union of arithmetic progressions, hence periodic. Let $\tau_\ell$ be the (smallest) period of $A\cap [\xi_\ell,\xi_{\ell+1})$, and denote its density by $\sigma_\ell:=|A\cap [\xi_\ell,\xi_{\ell}+\tau_\ell)|/\tau_\ell$. By the above observations, cf. \eqref{eq:observ1}, \eqref{eq:observ2}, and \eqref{eq:observ3}, we have 
$$
|\sigma_\ell-\alpha|\le \varepsilon_{\kappa(\ell)}
\quad \text{ and }\quad 
\tau_\ell \le \varepsilon_{\kappa(\ell)}\xi_\ell. 
$$
Since $\lim_\ell \kappa(\ell)=\infty$, we have also $\lim_\ell \varepsilon_{\kappa(\ell)}=0$. 

\medskip

At this point, for each integer $N\ge \xi_1$, let $g(N)$ be the unique positive integer such that $\xi_{g(N)} \le N<\xi_{g(N)+1}$. It easily follows that 
\begin{displaymath}
    \begin{split}
        |A\cap [0,N)|&=\sum_{\ell=1}^{g(N)-1}|A \cap [\xi_\ell,\xi_{\ell+1})|+|A \cap [\xi_{g(N)},N)|\\
        &=\sum_{\ell=1}^{g(N)-1}\sigma_\ell(\xi_{\ell+1}-\xi_\ell)+\sigma_{g(N)}(N-\xi_{g(N)})+\Psi_N, 
    \end{split}
\end{displaymath}
for some real $\Psi_N$ such that $|\Psi_N|\le 2\sum_{\ell=1}^{g(N)}\tau_\ell$. Here, we also used that $A\cap [0,\xi_1)=\emptyset$. 
Taking into account \eqref{eq:observ1} and \eqref{eq:observ11} and that $N=\xi_1+\sum_{\ell=1}^{g(N)-1}(\xi_{\ell+1}-\xi_\ell)+(N-\xi_{g(N)})$, we obtain that 
\begin{displaymath}
    \begin{split}
\left||A\cap [0,N)|-\alpha N\right|
&\le \alpha\xi_1+\sum_{\ell=1}^{g(N)-1}|\sigma_\ell-\alpha|(\xi_{\ell+1}-\xi_\ell)+|\sigma_{g(N)}-\alpha|(N-\xi_{g(N)})+|\Psi_N|\\
&\le O(1)+\sum_{\ell=1}^{g(N)-1}\varepsilon_{\kappa(\ell)}(\xi_{\ell+1}-\xi_\ell)+\varepsilon_{\kappa(g(N))}(N-\xi_{g(N)})+2\sum_{\ell=1}^{g(N)}\tau_\ell.
    \end{split}
\end{displaymath}

Fix $\varepsilon>0$ and pick $\ell_0\ge 1$ such that $\varepsilon_{\kappa(\ell)}\le \varepsilon$ for all $\ell\ge \ell_0$. Since $\tau_\ell \le \varepsilon_{\kappa(\ell)} \xi_\ell$ and $\xi_{\ell+1}\ge 2\xi_\ell$ for all $\ell\ge 1$, we get 
$$
\sum_{\ell=\ell_0}^{g(N)}\tau_\ell \le \varepsilon \sum_{\ell=\ell_0}^{g(N)}\xi_\ell 
\le \varepsilon \xi_{g(N)}\sum_{k=0}^\infty 2^{-k}=2\varepsilon \xi_{g(N)}.
$$

Putting everything together, it follows that 
\begin{displaymath}
    \begin{split}
\left||A\cap [0,N)|-\alpha N\right|
&\le O(1)+\sum_{\ell=\ell_0}^{g(N)-1}\varepsilon_{\kappa(\ell)}(\xi_{\ell+1}-\xi_\ell)+\varepsilon_{\kappa(g(N))}(N-\xi_{g(N)})+2\sum_{\ell=\ell_0}^{g(N)}\tau_\ell\\
&\le O(1)+\varepsilon \left( \sum_{\ell=\ell_0}^{g(N)-1}(\xi_{\ell+1}-\xi_\ell)+(N-\xi_{g(N)})\right)+4\varepsilon \xi_{g(N)}\\
&\le O(1)+\varepsilon (N-\xi_{\ell_0})+4\varepsilon N\\
&\le O(1)+5\varepsilon N.
    \end{split}
\end{displaymath}
Since $\varepsilon$ was arbitrary, dividing by $N$ we conclude that $A \in \mathcal{D}$ and $\mathsf{d}(A)=\alpha$. 

\bigskip

\textbf{Upper density of $A+A$.} 
By the definition of $A$, we have that $B_m=(H_m\cdot\mathbb Z+R_m)\cap[x_m,y_m)\subseteq A$. 
By \eqref{eq:observ1} and the choice of $y_m$, we have $H_m\leq\varepsilon_m x_m$ and $x_m\leq\varepsilon_m y_m$. This implies that 
$$
2H_m\le 2\varepsilon_m^2y_m\le \frac{1}{8}y_m.
$$
On the other hand, we have also 
$$
y_m-x_m
\geq
(1-\varepsilon_m)y_m
>
\frac{3}{4}y_m.
$$
Therefore $y_m-x_m>2H_m$. Moreover, by \eqref{eq:choiceRm}, we have $R_m+R_m=\mathbb Z/H_m\mathbb Z$.  
Therefore, Lemma~\ref{lem:loudpiece111111} gives
$$
[2x_m+2H_m,2y_m-2H_m)\cap\mathbb N
\subseteq
B_m+B_m
\subseteq
A+A.
$$
Observe that $x_m/y_m\le \varepsilon_m$ and that $H_m/y_m \le \varepsilon_mx_m/y_m \le \varepsilon_m^2$. 
Hence, for each $m\ge 1$, it follows that 
\begin{displaymath}
    \begin{split}
        \frac{|(A+A)\cap[0,2y_m)|}{2y_m}&\ge \frac{|[2x_m+2H_m,2y_m-2H_m)|}{2y_m}\\
        &=1-\frac{x_m}{y_m}-\frac{2H_m}{y_m}\\
        &\ge 1-\varepsilon_m-2\varepsilon_m^2 \\
        &\ge 1-3\varepsilon_m.
    \end{split}
\end{displaymath}
Since $\lim_m \varepsilon_m=0$, we conclude that $\mathsf{d}^\star(A+A)=1$. 
 
\bigskip

\textbf{Lower density of $A+A$.} 
Fix $m\geq 1$, and write
$$
F_m:=A\cap[0,z_{m,r})
\quad\text{ and }\quad
G_m:=C_{m,r}=A\cap[z_{m,r},x_{m+1}).
$$
Since $A\cap [0,x_{m+1})=F_m\cup G_m$, it follows that 
\begin{equation}\label{eq:inclusion1}
(A+A)\cap[0,x_{m+1})
\subseteq
(F_m+F_m)\cup(F_m+G_m)\cup(G_m+G_m).
\end{equation}
Now, we claim that 
\begin{equation}\label{eq:finallowerdensity}
\frac{|(A+A)\cap[0,x_{m+1})|}{x_{m+1}}
\leq
2\varepsilon_{m+1}
+
(r+1)
\left(
2\delta+\frac{\varepsilon_m}{r}
\right)^{r-1}
+
(r+1)\varepsilon_m\varepsilon_{m+1}.
\end{equation}

\medskip

First, since $F_m\subseteq[0,z_{m,r})$, then 
\begin{equation}\label{eq:inclusion2}
\frac{|(F_m+F_m)|}{x_{m+1}}
\leq
\frac{2z_{m,r}}{x_{m+1}} \le 2\varepsilon_{m+1}.
\end{equation}

\medskip

Second, recall that 
$$
G_m=C_{m,r}=\bigcap_{i=1}^r(q_{m,i}\cdot \mathbb{N}+W_{q_{m,i}}) \cap [z_{m,r}, x_{m+1}),
$$
where $W_{q_{m,i}}=\{0,1,\ldots,\lfloor \delta q_{m,i}\rfloor\}$ for each $i \in \{1,2,\ldots,r\}$. Notice that $2\lfloor \delta q_{m,i}\rfloor\le 2\delta q_{m,i}<q_{m,i}$. For notational convenience, set also 
$$
D_{m,i}:=W_{q_{m,i}}+W_{q_{m,i}}=\{0,1,\ldots,2\lfloor \delta q_{m,i}\rfloor\}
$$
for each $i \in \{1,2,\ldots,r\}$. 
Since $q_{m,i}>r/\varepsilon_m$ and $\varepsilon_m,\delta \in (0,1/4)$ and $r\ge 2$, we get
$$
\frac{|D_{m,i}|}{q_{m,i}}
\leq
\frac{2\delta q_{m,i}+1}{q_{m,i}}
=2\delta+\frac{1}{q_{m,i}}
<
2\delta+\frac{\varepsilon_m}{r}=:\theta_m \in (0,1)
$$
for each $i \in \{1,2,\ldots,r\}$. Setting 
$$
\Gamma_{m}
:=
\bigcap_{i=1}^r(q_{m,i}\cdot \mathbb{N}+D_{m,i})
\quad \text{ and }\quad 
Q_m:=q_{m,1}q_{m,2}\cdots q_{m,r},
$$
it follows that 
\begin{equation}\label{eq:inclusion3}
G_m+G_m\subseteq \Gamma_m
\end{equation}
and that, by the Chinese remainder theorem, 
\begin{equation}\label{eq:inclusion4}
\frac{|\Gamma_m \cap [0,x_{m+1})]|}{x_{m+1}}\le \theta_m^r+\frac{Q_m}{x_{m+1}} \le \theta_m^{r-1}+\varepsilon_m\varepsilon_{m+1}.
\end{equation}
In the last inequality, we also used the facts that $\theta_m \in (0,1)$ and, by construction, 
$$
Q_m\le \pi_mQ_m \le \varepsilon_mz_{m,r}\le \varepsilon_m \varepsilon_{m+1}x_{m+1}.
$$

\medskip

Lastly, suppose that $x+y \in (A+A)\cap [0,x_{m+1})$ with $x \in F_m$ and $y \in G_m$. 
If $x<y_m=z_{m,0}$ then $x \in W_{q_{m,i}}$ for each $i \in \{1,\ldots,r\}$ because 
$$
\max W_{q_{m,i}}=\lfloor \delta q_{m,i}\rfloor \ge 2z_{m,i-1} >z_{m,0}
$$
for each $i \in \{1,\ldots,r\}$. Hence $x+y \in q_{m,i}\cdot \mathbb{N}+D_{m,i}$ for all $i \in \{1,\ldots,r\}$, i.e., $x+y \in \Gamma_m$. 

Otherwise, fix now $k \in \{0,1,\ldots,r-1\}$ and suppose that $x \in C_{m,k}$ and $y \in G_m=C_{m,r}$. Define also 
$$
\Gamma_{m,k}:=\bigcap_{\substack{1\le i\le r,\\ i\neq k+1}}(q_{m,i}\cdot \mathbb{N}+D_{m,i}).
$$

If $i\leq k$, then the definition of $C_{m,k}$ gives $x \in q_{m,i}\cdot \mathbb{N}+W_{q_{m,i}}$. 
Since the same is true for $y$, we obtain $x+y \in q_{m,i}\cdot \mathbb{N}+D_{m,i}$.

If $i\geq k+2$, then
$$
\max W_{q_{m,i}}=\lfloor \delta q_{m,i}\rfloor \geq
2z_{m,i-1}
>z_{m,i-1} \ge z_{m,k+1}>x. 
$$
Thus 
$x \in q_{m,i}\cdot \mathbb{N}+W_{q_{m,i}}$
and again $x+y\in q_{m,i}\cdot \mathbb{N}+D_{m,i}$. 

Therefore, the only congruence condition which may fail for a sum in $C_{m,k}+G_m$ is the condition modulo $q_{m,k+1}$. This proves that 
\begin{equation}\label{eq:inclusion5}
F_m+G_m\subseteq \Gamma_m \cup \bigcup_{j=0}^{r-1}\Gamma_{m,j}.
\end{equation}
Reasoning analogously as before, for each $j \in \{0,1,\ldots,r-1\}$, we have 
\begin{equation}\label{eq:inclusion6}
\frac{|\Gamma_{m,j} \cap [0,x_{m+1})]|}{x_{m+1}}\le \theta_m^{r-1}+\frac{Q_m}{x_{m+1}} \le \theta_m^{r-1}+\varepsilon_m\varepsilon_{m+1}.
\end{equation}

It follows by \eqref{eq:inclusion1}, \eqref{eq:inclusion2}, \eqref{eq:inclusion3}, \eqref{eq:inclusion5}, and the upper estimates \eqref{eq:inclusion4} and \eqref{eq:inclusion6} that 
\begin{displaymath}
    \begin{split}
        \frac{|(A+A)\cap[0,x_{m+1})|}{x_{m+1}}
&\leq
\frac{|F_m+F_m|}{x_{m+1}}+\frac{|\Gamma_m \cap [0,x_{m+1})|}{x_{m+1}}+\sum_{j=0}^{r-1}\frac{|\Gamma_{m,j} \cap [0,x_{m+1})|}{x_{m+1}}\\
&\leq 2\varepsilon_{m+1}+(r+1)\theta_m^{r-1}+(r+1)\varepsilon_m\varepsilon_{m+1}.
    \end{split}
\end{displaymath}
This proves inequality \eqref{eq:finallowerdensity}. Taking the lower limit along the sequence $(x_{m+1})_{m\geq1}$, we conclude that $\mathsf{d}_\star(A+A)\le (r+1)(2\delta)^{r-1}$. 
\end{proof}

\medskip

\begin{proof}
[Proof of Corollary \ref{cor:ruzsa}]
    Fix constants $\nu \in (0,1)$ and $c>0$. 
    Pick a sufficiently large integer $r\ge 2$ such that $(r-1)/r>\nu$. Since $r-1-r\nu>0$, then $x^{r-1-r\nu}\to 0$ as $x\to 0^+$. Hence it is possible to choose $\delta \in (0,\nicefrac{1}{4})$ such that 
    $$
    \delta^{r-1-r\nu}<\frac{c}{(r+1)2^{r-1}}. 
    $$
    Thanks to Theorem \ref{thm:construction}, there exists $A \in \mathcal{D}$ which satisfies \eqref{eq:mainclaim}. It follows that 
    $$
    \mathsf{d}_{\ast}(A+A)
    \le (r+1)2^{r-1}\delta^{r-1}
    < c \cdot 1 \cdot \delta^{r\nu}=c\bigl(\mathsf{d}^{\ast}(A+A)\bigr)^{1-\nu}\mathsf{d}(A)^{\nu}.
    $$
    Therefore Inequality \eqref{eq:claimedRuzsa} fails. 
\end{proof}

\medskip

\begin{proof}
    [Proof of Proposition \ref{prop:optimality}] 
Fix $A \in \mathcal{D}$ with $x:=\mathsf{d}(A)>0$, and define 
$$
y:=\mathsf{d}_{\star}(A+A)
\quad \text{ and }\quad 
z:=\mathsf{d}^{\star}(A+A).
$$
Since $A+A$ contains a translation of $A$, then $0<x\le y\le z$. In particular, $\mathscr{M}_q(x,z)$ is well defined and 
$
x\le \mathscr{M}_q(x,z)\le z.
$

Thanks to Kneser's theorem \cite{MR56632}, if $y<2x$, then $A+A$ is eventually periodic, hence $A+A\in\mathcal{D}$ and $y=z$. In such case,
$$
y=z\ge \mathscr{M}_q(x,z)
\ge c_q\mathscr{M}_q(x,z).
$$

Hence, suppose hereafter that $y\ge 2x$. If $z\le 2x$, then
$$
y\ge 2x
\ge \mathscr{M}_q(2x,z)
\ge \mathscr{M}_q(x,z)
\ge c_q\mathscr{M}_q(x,z).
$$
In the opposite, if $z>2x$, it follows by the hypothesis $q<0$ that
$$
\mathscr{M}_q(x,z)
=
\left(
\frac{x^q+z^q}{2}
\right)^{1/q}
\le
\left(
\frac{x^q}{2}
\right)^{1/q}
=
2^{-1/q}x.
$$
Therefore 
$$
y
\ge 
2x\ge 
2^{1+1/q}\mathscr{M}_q(x,z)
\ge
c_q\mathscr{M}_q(x,z).
$$
This completes the proof.
\end{proof}

\bibliographystyle{amsplain}
\bibliography{densityrefs}

\end{document}